\documentclass[12pt, oneside,reqno]{amsart}
\pdfoutput=1

\usepackage[a4paper,margin=1in]{geometry}
\usepackage{amsmath,amssymb,amsthm,mathtools}
\usepackage{graphicx}
\usepackage{subcaption}
\usepackage{float}
\usepackage{tikz}

\usepackage{appendix}

\usepackage{bbold}

\usepackage[normalem]{ulem}

\usepackage{chngcntr}
\usepackage{longtable}

\usepackage{comment}

\usepackage{hyperref}
\hypersetup{colorlinks,urlcolor=blue,citecolor=blue}

\def\hhref#1{\href{http://arxiv.org/abs/#1}{arXiv:#1}} 

\usepackage[nameinlink,noabbrev]{cleveref}
\numberwithin{figure}{subsection}

\newcommand{\C}{\mathbb{C}}

\newcommand{\Z}{\mathbb{Z}}

\newcommand{\Q}{\mathbb{Q}}

\newcommand\Qt{Q_1}

\theoremstyle{plain}
\newtheorem{thm}{Theorem}[section]     
\newtheorem{lemma}[thm]{Lemma}

\theoremstyle{definition}
\newtheorem{defn}[thm]{Definition}
\newtheorem{remark}[thm]{Remark}
\newtheorem{ex}[thm]{Example}

\DeclareMathOperator{\ord}{ord}

\author{Ovidiu Costin$^{*}$, Gerald V. Dunne$^{\dagger}$, Ali Saraeb$^{\ddagger}$}

\address{$^{*}$Mathematics Department, The Ohio State University, 231 W. 18th Avenue,  Columbus, Ohio 43210, USA}
\address{$^{\dagger}$Physics Department, University of Connecticut, 196 Auditorium Road,
Storrs, CT 06269, USA}
\address{$^{\ddagger}$Mathematics Department, The Ohio State University, 231 W. 18th Avenue,  Columbus, Ohio 43210, USA}

\email{costin@math.osu.edu}
\email{gerald.dunne@uconn.edu}
\email{saraeb.1@osu.edu}

\title{On uniqueness of mock theta functions}

\begin{document}

\begin{abstract}
We develop a resurgent approach to the problem of unique continuation of mock theta functions across their natural boundary. The starting point is the representation of the associated Mordell-Appell integrals as Laplace transforms of resurgent functions, which serve as the primary analytic objects. 

By rotating the Laplace contour by $\pi$, i.e.\ onto the Stokes line, one obtains, in all known cases, the mock-modular relations between the Mordell-Appell integrals and the corresponding unary series in $\hat q=e^{-\pi i \tau}$ and $\hat q_1=e^{-\pi i (-1/\tau)}$. 

We then prove that these relations admit a unique solution on the $q$-side, expressed in terms of $q=e^{\pi i \tau}$ and $q_1=e^{\pi i (-1/\tau)}$, with coefficients determined by the corresponding Mordell-Appell integrals. This yields a canonical continuation across the natural boundary, given by a resurgent extension of the classical principle of permanence of relations, and singles out a distinguished family of mock theta functions in each group.

We present a complete analysis for the order $3$ and $5$ cases (mf$_3$ and mf$_5$). The method extends naturally to higher orders;   a general theory will appear in a separate paper.

\end{abstract}

\maketitle

\section{Introduction and Motivation}

Mock modular symmetry has had a profound impact in both mathematics and physics. The early work of Ramanujan and Watson \cite{Watson,GM12} has found applications in combinatorics \cite{ono}, number theory \cite{Zwe08,Zag09,folsom}, representation theory \cite{Bring15,Cheng:2022rqr}, low dimensional topology \cite{GM21},  statistical physics \cite{baxter,andrewsHH},  Chern-Simons theory \cite{lawrence,Hikami04}, Umbral moonshine \cite{CDH12}, and quantum black holes \cite{DMZ12,Adams:2026ash}. The connection between mock modular symmetry and Chern-Simons theory has recently been explored \cite{GMP,CCKPS3d,Cheng:2018vpl} using the theory of resurgent asymptotics \cite{ecalle,costin-book,Sauzin06}. 

An important challenge is to continue uniquely across a natural boundary. 
In \cite{CDGG,ACDGO} OC, GD and collaborators proposed that the extreme rigidity of resurgent functions under the operations of analysis provides a new form of unique continuation, called {\it resurgent continuation}, which generalizes  the operation of continuation to resurgent transseries. This has been applied to the natural boundary in Chern-Simons theory, a topological quantum field theory, for which crossing the natural boundary corresponds to the operation of orientation reversal of the underlying 3 dimensional manifold, thereby connecting topological invariants of orientation reversed manifolds  \cite{CDGG,ACDGO}.
 
Previous uniqueness proofs by two of the authors relied on Hauptmoduls or on special features of low-order mock theta functions \cite{CDGG}, and do not readily extend to higher-order cases. In this work, we present a method, introduced by AS, that yields independent proofs of uniqueness for several examples arising in the Chern–Simons context, using only elementary tools from analysis and number theory. This approach is amenable to extension to a broader class of cases \cite{inprep}.
  
\subsection{Resurgent continuation through the natural boundary.}

The resurgent continuation used in this work rests on two guiding ideas. 
The first is the \emph{rigidity of resurgence}, which plays a role analogous 
to the rigidity of analyticity: once the resurgent data (that is, the 
singularity structure of the Borel transform) are fixed, the analytic 
continuation of the associated Laplace integrals is uniquely determined. 

The second is a resurgent analogue of the classical \emph{permanence of 
relations}: algebraic or functional relations persist under resurgent continuation.

In the resurgent approach, the analytic objects underlying the mock theta
functions are not the $q$-series themselves, but the {\em Mordell-Appell
integrals from which they arise}. Indeed, after suitable changes of variables,
these integrals can be written as Laplace transforms of resurgent kernels in
the Borel plane, and therefore define resurgent functions in the original
variable \cite{CDGG,ACDGO}. Their analytic behavior is thus governed by the singularity
structure of the corresponding Borel-plane kernels.

A crucial feature is that the Mordell-Appell integrals satisfy transformation laws under the action of the relevant subgroup of $SL_2(\mathbb{Z})$, while at the same time {\em admitting analytic continuation through the unit circle in the variable $q$}. Writing
\[
q=e^{-\alpha},
\]
this corresponds to analytic continuation of the integrals from $\alpha>0$ to $\alpha<0$.

In the resurgent framework the natural place to analyze such analytic continuation is near {\em Stokes lines}, where the transseries associated with the Laplace integral changes. In the present situation this occurs as $\alpha$ approaches $\mathbb{R}^-$. Denoting
\[
\hat q=e^{\alpha}=q^{-1}, 
\qquad 
\hat q_1=e^{\pi^2/\alpha},
\]
we observe that when $\alpha\in\mathbb{R}^-$ both $\hat q$ and $\hat q_1$ are small. The analytically continued integrals then admit a canonical decomposition
into a $\hat q$-series and a $\hat q_1$-series. 

This decomposition arises
directly from elementary  residue calculus applied to the Laplace representation and ultimately corresponds to the
unique splitting of the integral into its real and imaginary parts. The real part is the \'Ecalle-Borel summed part of the transseries, and it equals to a purely real $\hat{q}$ series, while the $\hat{q_1}$ series is purely imaginary, and it is the exponential part of the transseries.

This calculation is presented for the Mordell integrals corresponding to arbitrary order (potential) mfs in \cite{CDGG,ACDGO}, and gives the relations among a selected group of the mfs (see \S\ref{S:select}) and the Mordell-Appell integrals in the known, classical cases. See also Remark \ref{rem:continuation} for an example.

From the point of view of resurgence, this decomposition is fully determined
by the singularity structure of the resurgent kernel, which in turn
governs the Stokes phenomenon of the Laplace transform. In particular, the continuation across the natural boundary is uniquely determined by { a resurgent Mordell-Appell integral. In other words, on the side where $|\hat q|<1$ and $|\hat q_1|<1$ (equivalently, $\alpha\in\mathbb{R}^-$), the modular relations are realized by an ensemble consisting of three components: a $\hat q$–series, a $\hat q_1$–series, and a Mordell–Appell integral. Together, these objects are closed under the action of the modular transformations and provide an analytic realization of the modular equations on this side of the boundary.

On the other hand, since the discrete isometries of the upper half plane $\mathbb{H}$ are represented by
$SL_2(\mathbb{Z})$, it is natural to regard the modular equations as
foundational: they play, for functions on the hyperbolic plane, a role
analogous to that played by difference equations for functions on
$\mathbb{R}$ or $\mathbb{C}$.  See also Appendix \ref{S:SL2Z}.
}

\subsection{Resurgent permanence of relations.}\label{S:select}

A striking observation is that in any given mf order,
the modular equations satisfied by the analytically continued integrals coincide with
those satisfied by their original $q,q_1$ decompositions {\em  { by exactly one set of mfs}}. This unique set, which inherits resurgence from the underlying Mordell-Appell integral \footnote{again, nonsingular for  $|q|=1$}, represents the unique continuation of the solution of the modular equations through the natural boundary of the mfs. This unique, equation-preserving, continuation is seen as the resurgent analogue of the principle of permanence of relations.

{ Furthermore, still by permanence of relations,} whenever
independent considerations from physics or topology assign meaning to
both the $q$-expansions and the $1/q$-expansions, the resulting
continuation is expected to coincide with the resurgent continuation.

In this work, we present uniqueness proofs for the continuation across the natural boundary for two foundational examples: order 3 and order 5 mock theta functions, which we refer to as $\text{mf}_3$ and $\text{mf}_5$. In our approach, resurgence plays a central role.\footnote{As the natural boundary is approached, the asymptotic expansions of the relevant mock theta functions are factorially divergent and consequently lie outside the Denjoy-Carleman quasianalytic regime of unique continuation \cite{Rudin}.}

\section{Notation and Definitions}\label{sec:notation}

\noindent Throughout the paper, we use the following notation, conventions, and terminology. 

\subsection{Basic notation}

We write $\mathbb H=\{\tau\in\mathbb C:\Im\tau>0\}$ for the upper half-plane, 
$\mathbb D=\{z\in\mathbb C:|z|<1\}$ for the unit disk, and $\mathbb{1}$ for the $2\times 2$ identity matrix. Throughout, $\tau\in\mathbb H$, and we set
\[
\alpha=-\pi i\tau, \qquad q=e^{\pi i\tau}=e^{-\alpha}, \qquad Q=q^2=e^{2\pi i\tau}.
\]
Under the modular involution $S:\tau\mapsto -1/\tau$, we write
\[
q_1=e^{\pi i(-1/\tau)}=e^{-\pi^2/\alpha}, \qquad Q_1=q_1^2=e^{2\pi i(-1/\tau)}.
\]
We use the $q$-Pochhammer symbols
\[
(a;b)_n=\prod_{j=0}^{n-1}(1-ab^j), \qquad (a;b)_\infty=\prod_{j=0}^\infty(1-ab^j),
\]
and denote by $\eta(\tau)$ the Dedekind eta function
\[
\eta(\tau)=Q^{1/24}(Q;Q)_\infty.
\]
We denote by $S$ and $T$ the standard generators of $SL_2(\mathbb Z)$,
\[
S:\tau\mapsto -1/\tau, \qquad T:\tau\mapsto \tau+1,
\]
and write $\Gamma_\theta=\langle S,T^2\rangle$ for the theta subgroup.

Finally, we use $f\lesssim g$ to mean $|f|\le C|g|$ for some constant $C>0$, and 
$f\asymp g$ to mean $f\lesssim g$ and $g\lesssim f$.
For $r\in\mathbb Q$, we define
\[
q^r := e^{\pi i r\tau}, \qquad Q^r := e^{2\pi i r\tau},
\]
and similarly, under $S:\tau\mapsto -1/\tau$,
\[
q_1^r := e^{\pi i r(-1/\tau)}, \qquad Q_1^r := e^{2\pi i r(-1/\tau)}.
\]
Given a subgroup $\Gamma\le SL_2(\mathbb Z)$, we say that a function $F:\mathbb H\to\mathbb C$ is invariant under $\Gamma$ if
\[
F(\gamma\tau)=F(\tau), \qquad \gamma\in\Gamma.
\]
We now recall the standard notions of a \emph{cusp}, the \emph{width} of a cusp, and \emph{meromorphy at a cusp} \cite[p. 5-7]{Stein}.

\begin{defn} \label{def:cusp}
Let \(\Gamma\leq SL_2(\Z)\) be a subgroup of finite index.
\begin{enumerate}
    \item A \emph{cusp} of \(\Gamma\) is a \(\Gamma\)-equivalence class in
    \(\Q\cup\{\infty\}\).

    \item Let \(\mathfrak a\in \Q\cup\{\infty\}\) be a cusp of \(\Gamma\), and
    choose \(\sigma_{\mathfrak a}\in SL_2(\Z)\) such that
    \(\sigma_{\mathfrak a}(\infty)=\mathfrak a\). We define the \emph{width} of the cusp
    \(\mathfrak a\) to be the smallest positive integer \(w_{\mathfrak a}\) such that
    \[
    T^{w_{\mathfrak a}}
    \in
    \sigma_{\mathfrak a}^{-1}\Gamma_{\mathfrak a}\sigma_{\mathfrak a},
    \qquad
    \Gamma_{\mathfrak a}:=\{\gamma\in\Gamma:\gamma\mathfrak a=\mathfrak a\}.
    \]

    \item A meromorphic \(\Gamma\)-invariant function \(F\) on \(\mathbb H\) is
    said to be meromorphic at the cusp \(\mathfrak a\) if
    \(F(\sigma_{\mathfrak a}\tau)\) admits a Laurent expansion in
    \[
    q_{\mathfrak a}:=\exp\!\Bigl(\frac{2\pi i\tau}{w_{\mathfrak a}}\Bigr)
    \]
    as \(\Im\tau\to\infty\), of the form
    \begin{equation}\label{eq:cusp-laurent}
    F(\sigma_{\mathfrak a}\tau)=\sum_{n\ge n_\mathfrak a} c_n q_{\mathfrak a}^{\,n},
    \qquad c_{n_\mathfrak a}\neq 0.
    \end{equation}
    We then define the order of $F$ at the cusp $\mathfrak a$ to be
    \[
    \ord_{\mathfrak a}(F)=n_\mathfrak a.
    \]
\end{enumerate}
\end{defn}
For more details, see \cite[Sections 2, 4]{MilneMF} and \cite{Paule}. 
\begin{remark}
Indeed, the coefficients $c_n$ in \eqref{eq:cusp-laurent} depend on the choice of $\sigma_{\mathfrak a}$. However, the width $w_{\mathfrak a}$ and the order $\ord_{\mathfrak a}(F)=n_{\mathfrak a}$ are well-defined, i.e. they do not depend on the choice of $\sigma_{\mathfrak a}$ (see, e.g., \cite[p.~27-29]{Bhattacharya}). 
\end{remark}

\begin{remark}
    If $\ord_{\mathfrak a}(F) \ge0$ for some cusp $\mathfrak a$,  we say $F$ is holomorphic at $\mathfrak a$.
\end{remark}

\begin{ex}
    It is easy to show that the action of $SL_2(\Z)$ on $\Q \cup \{\infty\}$ has only one orbit, namely $\infty,$ and thus the only cusp of $SL_2(\Z)$ is $\infty$ \cite[p. 5]{Stein}.
\end{ex}

\begin{ex} \label{ex:theta-group}
The theta group\protect\footnotemark
\begin{equation}
    \Gamma_\theta:=\langle\, S,T^2\,\rangle,\qquad S:\tau\mapsto-1/\tau,\quad T^2:\tau\mapsto\tau+2.
\end{equation}
has two cusps: $1$ and $\infty$ (note that $-1\sim1$ under $T^2$), since one can show that every rational number is $\Gamma_\theta-$equivalent to either $1 \, (=1/1)$ or $\infty \, (=1/0)$ (see e.g., \cite[p.~ 20-29]{Schultz} or \cite[p.~ 76, 77]{GottscheZagier}). It is also standard and easy to show that the cusps $1$ and $\infty$ of $\Gamma_\theta$ have widths $w_1=1$ and $w_\infty=2,$ respectively (see, for example, \cite[p.~65]{maass} or apply Stein's algorithm \cite[p.~ 9]{Stein}). 
\footnotetext{The principal congruence subgroup of $SL_2(\mathbb{Z})$ of level $2$ is defined by
\begin{equation}
\Gamma(2):=
\left\{
\begin{pmatrix}
a & b\\
c & d
\end{pmatrix}
\in SL_2(\mathbb{Z})
:\ 
\begin{pmatrix}
a & b\\
c & d
\end{pmatrix}
\equiv
\begin{pmatrix}
1 & 0\\
0 & 1
\end{pmatrix}
\pmod{2}
\right\}.
\end{equation}
The theta group can be alternatively defined as the union of two cosets of $\Gamma(2)$, namely $\Gamma_\theta:= \Gamma(2) \,\cup\, S\, \Gamma(2).$ It is an important group in the theory of modular forms, satisfying $\Gamma(2)\subset \Gamma_\theta \subset \Gamma(1):=SL_2(\mathbb{Z}),
$ and $[\Gamma(1):\Gamma_\theta]=3.$ For more details, see \cite[p.~ 20-29]{Schultz} or \cite[p.~ 76]{GottscheZagier}.}
\end{ex}

\begin{defn}
Let \(\Gamma\leq SL_2(\Z)\) be a subgroup of finite index, and let \(k\in\Z\).

\begin{enumerate}
    \item A holomorphic function $F:\mathbb H\to\C$ is called a \emph{modular form of weight $k$ for $\Gamma$} if
    \[
    F(\gamma\tau)=(c\tau+d)^k F(\tau)
    \qquad
    \forall\ 
    \gamma=\begin{pmatrix}a&b\\c&d\end{pmatrix}\in\Gamma,\ \tau\in\mathbb H,
    \]
    and if $F$ is holomorphic at every cusp of $\Gamma$.
    \item A \emph{meromorphic modular form of weight $k$ for $\Gamma$} is defined as in $(1),$ except that $F$ is assumed to be meromorphic on $\mathbb H$ and meromorphic at every cusp of $\Gamma$.
    \item A meromorphic function $F:\mathbb H\to\C$ is called a \emph{modular function for $\Gamma$} if $F$ is invariant under $\Gamma$ and meromorphic at every cusp of $\Gamma$.
\end{enumerate}
\end{defn}

For the sake of clarity, we start with the proof of uniqueness for mf$_5$, since the conceptual structure of the transformations is simpler, namely all of them can be written in terms of $Q$ and $Q_1$. Furthermore, mf$_3$ introduces a new method for the more involved link $-q\mapsto-q_1$. 
\begin{remark}
Every holomorphic modular function is constant. For example, if $F$ is invariant under $SL_2(\Z)$ and holomorphic on $\mathbb H$ and at the cusp $\infty,$ it descends to a holomorphic function on the compact Riemann surface $SL_2(\mathbb Z)\backslash(\mathbb H \cup \Q \cup\{\infty\})$ and therefore is a constant.
\end{remark}

\section{Uniqueness proof for mf\texorpdfstring{$_5$}{5}}
\label{sec:mf$_5$}

In this section, we prove the uniqueness of the Mock 5 identities listed below in \eqref{eq:transf-chi} and \eqref{eq:mf5-identity}.

\subsection{The transformation laws} \label{sec: TL mf5}
Consider two of Ramanujan's original order 5 mock theta functions, denoted $\chi_0$ and $\chi_1$, \cite{Watson,GM12}, defined as
\begin{eqnarray}
    \chi_0(q):= \sum_{n=0}^\infty \frac{q^n}{(q^{n+1}; q)_n}
    \qquad, \qquad
    \chi_1(q):= \sum_{n=0}^\infty \frac{q^n}{(q^{n+1}; q)_{n+1}}
\end{eqnarray}
The relevant Mordell-Appell integrals are expressed in terms of (\cite{GM12}, page 118):
\begin{align}
L(r,\alpha)
&=\int_{0}^{\infty} e^{-\frac{3}{2}\alpha x^{2}}\, 
  \frac{\cosh \bigl((3r-2)\alpha x\bigr)+\cosh \bigl((3r-1)\alpha x\bigr)}
       {\cosh(\frac{3}{2}\alpha x)}\,dx.
\label{eq:L def}
\end{align}
The order-5 mock theta functions $\chi_0,\chi_1$ satisfy the following mock modular transformation laws for ${\rm Re}(\alpha)>0$  \cite[p.\ 118]{GM12} 
\begin{subequations}\label{eq:transf-chi}
\begin{align}
q^{-1/120}\bigl(\chi_0(q)-2\bigr)
= &-\sqrt{\frac{\pi(5-\sqrt5)}{5\alpha}}\,q_1^{-1/30}\bigl(\chi_0(q_1^{4})-2\bigr)
   -\sqrt{\frac{\pi(5+\sqrt5)}{5\alpha}}\,q_1^{71/30}\chi_1(q_1^{4}) \\ \nonumber
   &-\sqrt{\frac{135\alpha}{2\pi}}\,L\!\left(\frac15,5\alpha\right),
\end{align}
\begin{align}
q^{71/120}\chi_1(q)
= &-\sqrt{\frac{\pi(5+\sqrt5)}{5\alpha}}\,q_1^{-1/30}\bigl(\chi_0(q_1^{4})-2\bigr)
   +\sqrt{\frac{\pi(5-\sqrt5)}{5\alpha}}\,q_1^{71/30}\chi_1(q_1^{4}) \\ \nonumber
   &-\sqrt{\frac{135\alpha}{2\pi}}\,L\!\left(\frac25,5\alpha\right).
\end{align}
\end{subequations}

The transformation laws in \eqref{eq:transf-chi} are more naturally 
expressed in terms of $Q$ and $Q_1$, where we recall the notation defined in section \ref{sec:notation}:
\begin{eqnarray}
    Q=q^2=e^{-2\alpha} \qquad; \qquad \Qt=q_1^2 = e^{-2\pi^2/\alpha}
    \label{eq:q2}
\end{eqnarray}
This makes it possible to write the transformation laws \eqref{eq:transf-chi} in a more compact and symmetric matrix form. Evaluating \eqref{eq:transf-chi} at $2\alpha$, we have
\begin{eqnarray} 
\label{eq:mf5-identity}
    \mathcal L(\alpha)= 
    \begin{pmatrix}
    Q^{-1/120} \, \mathcal K_0(Q) \cr 
    Q^{-49/120} \, \mathcal K_1(Q)  
    \end{pmatrix}
    +\sqrt{\frac{\pi}{\alpha}} \, M\, 
    \begin{pmatrix}
    \Qt^{-1/120} \, \mathcal K_0(\Qt) \cr 
    \Qt^{-49/120} \, \mathcal K_1(\Qt)  
    \end{pmatrix}
    \quad, \quad \Re\alpha>0
\end{eqnarray}
Here we have defined
\begin{eqnarray}
    \mathcal K_0(Q)=2-\chi_0(Q) \qquad; \qquad \mathcal K_1(Q)=-Q\, \chi_1(Q)
    \label{eq:mf$_5$-solution}
\end{eqnarray}
and $\mathcal L (\alpha)$ in \eqref{eq:mf5-identity} is the 2-component vector of Mordell-Appell integrals from \eqref{eq:L def}:
\begin{eqnarray}
    \mathcal L (\alpha)=\sqrt{\frac{135\,\alpha}{\pi}}\, \begin{pmatrix} L\left(1/5, 10\alpha\right)
    \cr  {L\left(2/5, 10\alpha\right)}
    \end{pmatrix}
    \label{eq:lint}
\end{eqnarray}
The matrix $M$ in \eqref{eq:mf5-identity} is the $2\times 2$ mixing matrix
\begin{eqnarray}
    M=\frac{2}{\sqrt{5}} \begin{pmatrix}
    \sin\left(\frac{\pi}{5}\right) & \sin\left(\frac{2\pi}{5}\right)
    \cr \cr
    \sin\left(\frac{2\pi}{5}\right) & -\sin\left(\frac{\pi}{5}\right)
    \end{pmatrix}
    \label{eq:m}
\end{eqnarray}

\begin{remark}
The map $\alpha\mapsto \alpha_1:=\pi^2/\alpha$ (equivalently $\tau\mapsto -1/\tau$) is an involution,
and we have $(Q_1)_1=Q$.
\end{remark}
\begin{remark}
    Note that $M^2=\mathbb{1}$ and $\det(M)=-1.$
\end{remark}
\begin{remark}
$M$ also characterizes the modular
transformation $\alpha\to \frac{\pi^2}{\alpha}$ of the vector of Mordell-Appell integrals in \eqref{eq:lint}
\begin{eqnarray}
\mathcal L (\alpha)= \sqrt{\frac{\pi}{\alpha}}\, M\, \mathcal L\left(\frac{\pi^2}{\alpha}\right)
\label{eq:l-vector}
\end{eqnarray}
which is a consistency condition for the identity \eqref{eq:mf5-identity}, since the modular transformation interchanges $Q$ and $\Qt$, and we also have  that $M^2=\mathbb{1}$.
\end{remark}

\begin{remark}
    It is important to note that in \eqref{eq:mf5-identity} the functions $\mathcal K_j(Q)$, for $j=0, 1$, are $Q$-series expansions in non-negative integer powers of $Q$, defined in \eqref{eq:q2}. Furthermore, we observe that for any pair of holomorphic $q$-series $(\chi_0,\chi_1)$, the corresponding pair $(\mathcal K_0, \mathcal K_1)$ satisfies
    \begin{equation} \label{eq: estimates}
        \mathcal K_0(Q)= O(1), \qquad \mathcal K_1(Q)=O(Q), \qquad Q \to 0.
    \end{equation}
    This follows from 
    the fact that $\chi_0$ and $\chi_1$ are holomorphic at $Q=0$. 
\end{remark}

\begin{remark}
\label{rem:continuation}
    The $Q$-series structure of $\mathcal K_j(Q)$ in \eqref{eq:mf5-identity} is completely determined by the behavior of the Mordell-Appell integrals in $\mathcal L(\alpha)$ on the Stokes line $\alpha <0$. 
  The contour of integration defining $\mathcal{L}$ can be rotated to the Stokes line, $\alpha\mapsto e^{\pm i (\pi-\varepsilon)}\alpha$, with $\varepsilon>0$ arbitrarily small, where the integrals have a unique decomposition into real and imaginary parts
    \cite{CDGG,ACDGO}:
\begin{eqnarray} 
    \mathcal L(\alpha)= 
    \begin{pmatrix}
    Q^{1/120} \, X_0(1/Q) \cr 
    Q^{49/120} \, X_1(1/Q)  
    \end{pmatrix}
    +i\, \sqrt{\frac{\pi}{|\alpha|}} \, M\, 
    \begin{pmatrix}
    \Qt^{1/120} \, X_0(1/\Qt) \cr 
    \Qt^{49/120} \, X_1(1/\Qt) 
    \end{pmatrix}
    \qquad, \qquad (\alpha<0)
    \label{eq:mf5-identity-stokes}
\end{eqnarray}
Here $X_0(1/Q)$ and $X_1(1/Q)$ are the {\it unary} $Q$-series valid for $|Q|>1$:
\begin{eqnarray}
    X_0\left(\frac{1}{Q}\right)&=& Q^{1/120}\sum_{k=0}^\infty (-1)^k \left(Q^{-(14-(2k+1)15)^2/120}+Q^{-(14+(2k+1)15)^2/120} \right.
    \nonumber\\
    && \hskip 2cm  \left. + Q^{-(4-(2k+1)15)^2/120}+Q^{-(4+(2k+1)15)^2/120}\right)
    \nonumber \\
    X_1\left(\frac{1}{Q}\right) &=& Q^{49/120}\sum_{k=0}^\infty (-1)^k \left(Q^{-(8-(2k+1)15)^2/120}+Q^{-(8+(2k+1)15)^2/120} \right.
    \nonumber\\
    && \hskip 2cm  \left. + Q^{-(2-(2k+1)15)^2/120}+Q^{-(2+(2k+1)15)^2/120}\right)
    \label{eq:unary}
\end{eqnarray}
These large $Q$ expansions of $X_0$ and $X_1$ are expressed in terms of (unary) false theta functions \cite{CDGG,ACDGO}.
These results follow from straightforward residue computations and the principal value result for $\Re t>0$:
\begin{eqnarray}
\sqrt{\frac{4p\, t}{\pi}}\, PV\int_0^\infty dx\, e^{-p x^2/t} \frac{\cos( a x)}{\cos(p x)}
=\sum_{k=0}^\infty (-1)^k \left(e^{-(a-(2k+1)p)^2/(4 p t)}+e^{-(a+(2k+1)p)^2/(4 p t)}\right)
\label{eq:pv}
\end{eqnarray}
     Changing variables in $\mathcal{L}$ so that it becomes a vector of Laplace transforms, \eqref{eq:mf5-identity-stokes} coincides with the (necessarily unique)  transseries decomposition as $\alpha\to 0$ ($1/Q\to 1$) of  $\mathcal{L}$ on the Stokes line, where the vector in $1/Q$, written canonically as a PV integral in \eqref{eq:pv}, is the \'Ecalle-Borel sum of the asymptotic series as $\alpha\to 0$, and the $1/Q_1$ vector is the purely exponential part of the transseries. This structure is therefore required to be preserved under resurgent continuation across the natural boundary of the vector functions in $1/Q$ and $1/Q_1$ \footnote{We note again that $\mathcal{L}$ is not singular on that boundary.} back to the $\Re \alpha >0$ region \cite{CDGG,ACDGO}. This means that the full structure of \eqref{eq:mf5-identity}, such as the rational exponents, the $\sqrt{\pi/\alpha}$ factor, and the mixing matrix $M$, are all fixed by the Mordell-Appell integrals $\mathcal L(\alpha)$. This matches exactly the known structure of the identity \eqref{eq:mf5-identity} and the solution in \eqref{eq:mf$_5$-solution} on the other side of the natural boundary where $\Re \alpha >0$. Here we prove using elementary methods that \eqref{eq:mf$_5$-solution} is indeed the {\it unique} holomorphic $Q$-$Q_1$ pair of vectors solving \eqref{eq:mf5-identity}.
    
\end{remark}

\begin{remark}
On the non-unary side, where $\Re \alpha>0$,
    other solutions to \eqref{eq:mf5-identity} are known. However, these other solutions do not match the structure coming from the Mordell-Appell integrals on the Stokes line, and furthermore they are not holomorphic in $Q$. See Appendix C.
\end{remark}

\begin{remark} 
It is pointed out by Zagier in \cite{Zag09} that the pair $\chi_0$, $\chi_1$ is the only pair that satisfies the full set of $SL_2(\mathbb Z)$ transformations. This is similar in spirit to the conditions imposed in resurgent continuation.
\end{remark}

\subsection{The equations and normalization} \label{sec:eqnorm}
\noindent Suppose $(\chi_0,\chi_1)$ and $(\widetilde{\chi}_0,\widetilde{\chi}_1)$ are two pairs of
holomorphic $q$-series satisfying \eqref{eq:transf-chi}, and let $(\mathcal K_0,\mathcal K_1)$ and $(\widetilde{\mathcal K}_0,\widetilde{\mathcal K}_1)$ be the associated pairs defined by \eqref{eq:mf$_5$-solution}.
Define
\begin{eqnarray}
    \Delta \mathcal K_0(Q):= \mathcal K_0(Q)-\widetilde{\mathcal K}_0(Q)
    \qquad, \qquad 
    \Delta \mathcal K_1(Q):= \mathcal K_1(Q)-\widetilde{\mathcal K}_1(Q)
\end{eqnarray}
Then subtracting the two instances of \eqref{eq:mf5-identity} gives the {\it homogeneous} system
\begin{eqnarray}
    \begin{pmatrix}
    Q^{-1/120} \, \Delta\mathcal K_0(Q) \cr 
    Q^{-49/120} \, \Delta\mathcal K_1(Q)  
    \end{pmatrix}
    +\sqrt{\frac{\pi}{\alpha}} \, M\, 
    \begin{pmatrix}
    \Qt^{-1/120} \, \Delta\mathcal K_0(\Qt) \cr 
    \Qt^{-49/120} \, \Delta\mathcal K_1(\Qt) 
    \end{pmatrix} 
    =0
    \label{eq:homog}
\end{eqnarray}
Our strategy for proving uniqueness of \eqref{eq:mf5-identity} is to prove that the homogeneous system in \eqref{eq:homog} has no nontrivial holomorphic solution.

It is convenient to further simplify the homogeneous system \eqref{eq:homog}, effectively removing the $\sqrt{\pi/\alpha}$ factor, by the following normalization.
Recall the Dedekind eta function from section \ref{sec:notation}:
\begin{equation}
    \eta(\tau)=Q^{1/24}\prod_{n\ge1}(1-Q^n)=Q^{1/24}(Q;Q)_\infty
\end{equation}
and its transformation laws \cite{Apostol}
\begin{equation}\label{eq:eta tlaws}
\eta(\tau+1)=e^{\pi i/12}\eta(\tau), \qquad \eta \left(-\frac{1}{\tau}\right)=\sqrt{-i\tau}\,\eta(\tau)
=\sqrt{\frac{\alpha}{\pi}}\,\eta(\tau).
\end{equation}
Dividing \eqref{eq:homog} by $\eta(\tau)$, applying \eqref{eq:eta tlaws}, and using the normalization
\begin{equation}\label{eq:H23 def}
\mathcal H_0(Q):=\frac{\Delta \mathcal K_0(Q)}{(Q;Q)_\infty},\qquad
\mathcal H_1(Q):=\frac{\Delta \mathcal K_1(Q)}{(Q;Q)_\infty}.
\end{equation}
we obtain the equivalent homogeneous system
\begin{eqnarray}
    \begin{pmatrix}
    Q^{-1/20} \, \mathcal H_0(Q) \cr 
    Q^{-9/20} \, \mathcal H_1(Q)  
    \end{pmatrix}
    =-  M\, 
    \begin{pmatrix}
    \Qt^{-1/20} \, \mathcal H_0(\Qt) \cr 
    \Qt^{-9/20} \, \mathcal H_1(\Qt)  
    \label{eq:mf5-identity2}
    \end{pmatrix} 
    \label{eq:homog2}
\end{eqnarray}
where $M$ is the same mixing matrix as in \eqref{eq:m}. The advantage of \eqref{eq:homog2} is that it is symmetric under $S: Q\to\Qt$ (recall $M^2=\mathbb{1}$).
\begin{remark}
    Since $Q \to (Q;Q)_\infty$ is holomorphic and  zero-free on 
    $\mathbb D,$ both $\mathcal H_0$ and $\mathcal H_1$
    are clearly holomorphic $Q$-series on $\mathbb D$.
\end{remark}

\begin{remark}
    Using \eqref{eq: estimates} for each solution pair and
    \[
    \frac{1}{(Q;Q)_\infty}= \sum_{n\ge0} p(n) Q^n =1+O(Q), \qquad Q \to 0
    \]
    we obtain 
    \begin{equation} \label{eq: estimates Hj}
        \mathcal H_0(Q)=O(1), \qquad \mathcal H_1(Q)=O(Q),  \qquad Q \to 0
    \end{equation}
    Hence, to prove uniqueness, it suffices to show that \eqref{eq:homog2}
    has no nontrivial holomorphic $Q$-series solution satisfying \eqref{eq: estimates Hj}.
    \end{remark}

\begin{remark}
    The substitution map $q \to q^2$ is injective on the ring $\C[[q]]$ of the formal power series in $q$ with coefficients in $\C$. Thus, the uniqueness of the pair $(\chi_0(q),\chi_1(q))$ follows once the uniqueness of the corresponding pair $(\chi_0(Q),\chi_1(Q))$ is established.
\end{remark}

\begin{remark}
    The mf$_5$ transformation laws \eqref{eq:mf5-identity} have a special structural feature that motivates both the substitution $q \to q^2$ and the assumptions in \Cref{thm: uniq mf$_5$}: namely, the equations for $\chi_j(q)$ relate the value at $q$ to the value at $q_1^4$. The reason this is special is that, upon the substitution, one obtains a set of transformation laws for the functions $\chi_j(Q)$ that relate $Q=Q(\tau)=q^2= e^{2\pi i \tau}$ to  $Q_1 :=q_1^2 = e^{2\pi i (-1/\tau)}.$ This describes the behavior of $\chi_j(Q)$ under the inversion matrix $S$ since $Q(-1/ \tau)= Q_1.$ Hence we have a complete description of the transformation behavior of $\chi_j(Q)$ under $SL_2(\Z)= \langle S, T\rangle,$ since $\chi_j(Q)$ are invariant under $T.$ It therefore suffices to understand the boundary behavior modulo the action of $SL_2(\Z),$ that is, at the unique cusp $\infty$ for $SL_2(\Z)$. Hence, the only boundary condition one needs is at $\tau=i\infty$, equivalently at $Q=0$.
\end{remark}

\begin{remark}
    In this sense, the uniqueness proof for $\text{mf}_5$ is simpler than the corresponding proof for $\text{mf}_3$ in Section \ref{sec:mf$_3$}.
\end{remark}

\subsection{Uniqueness for mf$_5$}\label{sec:H23 proof}

\begin{thm} \label{thm: uniq mf$_5$}
The pair of order $5$ mock theta functions $(\chi_0,\chi_1)$ is the unique pair of functions holomorphic in $\mathbb D$
satisfying \eqref{eq:transf-chi}.
\end{thm}
In view of the discussion in Section \ref{sec:eqnorm}, the proof of this theorem follows from Lemma \ref{prop:H23-vanish} below.

\begin{lemma}\label{prop:H23-vanish}
Let $\mathcal H_0(Q)$ and $\mathcal H_1(Q)$ be holomorphic $Q$-series on $\mathbb{D}$ satisfying
\eqref{eq:homog2} and \eqref{eq: estimates Hj}. Then $\mathcal H_0\equiv 0$ and $\mathcal H_1\equiv 0$.
\end{lemma}

\begin{proof}
Define the holomorphic functions on $\mathbb H$ 
\begin{equation}\label{eq:v23 def}
v_0(\tau):=Q(\tau)^{-1/20}\mathcal H_0(Q(\tau)),\qquad
v_1(\tau):=Q(\tau)^{-9/20}\mathcal H_1(Q(\tau)),
\qquad
v(\tau):=\binom{v_0(\tau)}{v_1(\tau)},
\end{equation}
and observe that \eqref{eq:homog2} can be written as
\begin{equation}\label{eq:v M}
v(\tau)=-M\,v \left(-\frac1\tau\right)
\end{equation}
where $M$ is in \eqref{eq:m}.
Then since $M^2=\mathbb{1}$ we have
\begin{equation}\label{eq:S v}
v \left(-\frac1\tau\right)=-M\,v(\tau).
\end{equation}
On the other hand, since $\mathcal H_0$ and $\mathcal H_1$ are $Q$-series, with $Q=e^{2\pi i \tau}$, we have $\mathcal H_j(Q(\tau+1))=\mathcal H_j(Q(\tau))$, and thus \eqref{eq:v23 def} gives
\begin{equation}\label{eq:T v}
v(\tau+1)=D\,v(\tau),
\qquad
D:=\begin{pmatrix}e^{-\pi i/10}&0\\ 0&e^{-9\pi i/10}\end{pmatrix}.
\end{equation}
Hence, we now have a vector-valued modular object $v(\tau),$ so we consider its Wronskian determinant and show that, after normalization by a suitable power of $\eta,$ it gives rise to a scalar modular object.\footnote{This is the classical modular Wronskian method \cite{FrancMason}.} Indeed, consider the Wronskian
\begin{equation}\label{eq:Wronskian}
W(\tau):=\det \Bigl( \bigl(v(\tau),v'(\tau)\bigr) \Bigr)=v_0(\tau)v_1'(\tau)-v_0'(\tau)v_1(\tau).
\end{equation}
where derivatives are taken with respect to $\tau.$ Now using \eqref{eq:T v}, we obtain 
\begin{equation}\label{eq:W T}
W(\tau+1)=\det\Bigl(D \cdot\bigl(v(\tau), v'(\tau) \bigr) \Bigr) =\det(D)\,W(\tau)=(-1)\,W(\tau).
\end{equation}
Similarly, differentiating \eqref{eq:S v}, we get
\begin{equation}\label{eq:W S}
W \left(-\frac1\tau\right)
=\det \Bigl( \bigl(-Mv(\tau),\,\tau^2 (-M)v'(\tau)\bigr) \Bigr)
=\tau^2\det(-M)\,W(\tau)
=-\tau^2 W(\tau).
\end{equation}
Therefore, it is easy to check using \eqref{eq:eta tlaws}, \eqref{eq:W T}, and \eqref{eq:W S} that
\begin{equation}\label{eq:F def}
G(\tau):=\frac{W(\tau)^3}{\eta(\tau)^{12}}
\end{equation}
is invariant under both $S:\tau\mapsto -1/\tau$ and $T:\tau\mapsto\tau+1$. Moreover, $G$ is holomorphic on $\mathbb H,$ and since one can easily show,  using \eqref{eq: estimates Hj}, that $G(\tau)=O(Q)$ as $\tau\to i\infty$, it follows that $G$ extends holomorphically to the cusp $\infty$ with $G(\infty)=0$. However, the only holomorphic modular functions for $SL_2(\Z)$ are constants, and since G is such a function with $G(\infty)=0,$ it follows that $G \equiv 0$ on $\mathbb H$. This in turn implies that $W(\tau)\equiv 0$ on $\mathbb H$.

If $\mathcal H_0\equiv 0$ then \eqref{eq:homog2} forces $\mathcal H_1\equiv 0$, and there is nothing to show. Thus, assume WLOG that $\mathcal H_0\not\equiv 0$. This implies $v_0\not\equiv 0.$ Consider now the function $r$ given by $r(\tau):=v_1(\tau)/v_0(\tau)$, and note that $\tau \to r(\tau)$ is meromorphic on $\mathbb H.$ Since $W(\tau)$ vanishes on $\mathbb H,$ it follows that 
\begin{equation}
  0=W(\tau)=v_0(\tau)^2\left(\frac{v_1(\tau)}{v_0(\tau)}\right)'= v_0(\tau)^2 \, r'(\tau),
\end{equation}
which implies that $r$ is constant on $\mathbb H$. 
However, $r$ must be identically $0$ on $\mathbb H,$ since from \eqref{eq:T v}, we have
\begin{equation}
    r(\tau+1)=\frac{e^{-9\pi i/10}v_1(\tau)}{e^{-\pi i/10}v_0(\tau)}
=e^{-4\pi i/5}\, r(\tau)
\end{equation}
Therefore, $r \equiv 0$. Hence $v_1\equiv 0,$ and \eqref{eq:homog2} then gives $v_0\equiv 0$.

\end{proof}

\begin{remark}
    We close this section by noting that, besides the transformation laws in \eqref{eq:transf-chi}, the pair $(\chi_0,\chi_1)$ also satisfies a second $2 \times 2$ system of transformation laws that relate $-q$ to $-q_1$ (see e.g., \cite[p.\ 118]{GM12}). In addition to the Wronskian method introduced above, the study of this system requires an analysis of the behavior at the relevant cusps (see Section \ref{sec:mf$_3$} for an illustration in the mf$_3$ case). The corresponding result for the $-q \to -q_1$ system will appear elsewhere \cite{inprep}.
\end{remark}

\section{Uniqueness proof for mf\texorpdfstring{$_3$}{3}}
\label{sec:mf$_3$}

The relevant order-3 mock theta functions (${\rm mf}_3$) for this section are \cite{Watson,GM12}
\begin{eqnarray}
    \omega(q)\,&:=&\,\sum_{n=0}^{\infty}\frac{q^{\,2n(n+1)}}{(q;q^2)_{n+1}^{\,2}}
    \quad,\qquad |q|<1 
 \label{eq: omeg def}
 \\
f(q)\,&:=&\,\sum_{n=0}^{\infty}\frac{q^{\,n^2}}{(-q;q)_{n}^{\,2}}
    \quad,\qquad |q|<1 
 \label{eq: f def}
\end{eqnarray}

These are related by the mock modular transformation laws \cite{Watson,GM12}
\begin{eqnarray}
q^{2/3}\,\omega(-q)
&=&
-\sqrt{\frac{\pi}{\alpha}}\,q_1^{2/3}\,\omega(-q_1)
+\sqrt{\frac{12\alpha}{\pi}}\,W_3(\alpha) 
\quad, \qquad \alpha>0
\label{eq:omega tlaw}
\\
q^{2/3}\omega(q)&=&\sqrt{\frac{\pi}{4\alpha}}\,q_1^{-1/12} f(q_1^2)\,-\,\sqrt{\frac{3\alpha}{\pi}}\,W_2 \Bigl(\frac{\alpha}{2}\Bigr)
\quad, \qquad \alpha>0
\label{eq:omega f law}
\end{eqnarray}
Here $W_3$ and $W_2$ are the following Mordell-Appell integrals (see, e.g., \cite[p.~100, 114]{GM12}):
\begin{eqnarray}
 W_3(\alpha)&:=&\int_0^\infty e^{-3 \alpha x^2}\,\frac{\sinh(\alpha x)}{\sinh(3\alpha x)}\,dx
 \quad , \qquad \Re\alpha>0
 \label{eq:w3}
 \\
W_2(\alpha)&:=&\int_0^\infty e^{-\frac32\alpha x^2}\,\frac{\cosh(\alpha x)}{\cosh(3\alpha x)}\,dx
\quad , \qquad \Re\alpha>0
\label{eq:w2}
\end{eqnarray}

\begin{remark}
\label{rem: omega f unique}
  In this section we prove that the system \eqref{eq:omega tlaw}, \eqref{eq:omega f law} has a unique holomorphic solution. In fact, the hypotheses needed for the uniqueness of $\omega$ are weaker  than the full assumption \eqref{eq:omega tlaw}, \eqref{eq:omega f law}.  We show that \eqref{eq:omega tlaw} has a unique solution $\omega$ subject to the growth condition resulting from \eqref{eq:omega f law}, see Lemma~\ref{lem: growth cond} below. Once uniqueness of $\omega$ is established,  $f$ is also clearly uniquely determined by  \eqref{eq:omega f law}.
\end{remark}

\begin{thm}\label{T:omega uniqueness}
There is a unique pair $(f,\omega)$ of functions  holomorphic in $\mathbb D$  satisfying the system \eqref{eq:omega tlaw} \eqref{eq:omega f law} for $q$, $q_1$ in $\mathbb{D}$. 
\end{thm}
\begin{proof}[Proof of Theorem \ref{T:omega uniqueness}]
\noindent In the following, fix  $\theta_0$ with $0<\theta_0<\pi/2$ and we assume $\alpha \in \C$ with $|\arg\alpha|\le \theta_0$.  
\begin{lemma} 
\label{lem: growth cond of omega}
  The function $\omega$ satisfies the growth bound
  \begin{equation}\label{eq:growth-omega}
     \omega(e^{-\alpha}) \lesssim  |\alpha|^{-1/2} ~|q_1|^{\,-1/12}\text{ as }\alpha \to 0  \text{ if }|\arg\alpha|\le \theta_0
  \end{equation}
\end{lemma}
\begin{proof}   
From the Mordell-Appell integral in \eqref{eq:w2}, we
find the trivial estimate $W_2 \Bigl(\frac{\alpha}{2}\Bigr)\lesssim |\alpha|^{-1/2}$ for $\alpha\to0$ with $|\arg\alpha|\le\theta_0,$ since the function $\frac{\cosh(z)}{\cosh(3z)}$ is uniformly bounded on the sector $\{z: |\arg z| \le \theta_0 \}$ (the bound depends only on $\theta_0$). The claim thus follows by applying the triangle inequality to the transformation law in \eqref{eq:omega f law}. Here $f$ is Ramanujan's mf$_3,$ in \eqref{eq: f def}, which satisfies $f(q_1^2)=1 +O(q_1^2) \lesssim 1$ as $\alpha \to 0$ with $|\arg \alpha|<\theta_0,$ since $q_1 \to 0$ as $\alpha \to 0$ with $|\arg \alpha|<\theta_0.$

\end{proof}
\begin{lemma} \label{lem: growth cond}
There is a unique holomorphic function $\omega:\mathbb{D}\to\mathbb{C}$ satisfying \eqref{eq:omega tlaw} and \eqref{eq:growth-omega}. 

\end{lemma}

\begin{remark}\label{R:unique f}
    The assumptions in Lemma \ref{lem: growth cond} are the boundary conditions associated with the transformation law \eqref{eq:omega tlaw}. Indeed, writing $\Omega(\tau):= \omega(-q(\tau)),$ we observe that \eqref{eq:omega tlaw} describes the behavior of $\Omega$ under the modular transformation $S:\tau\to -1/\tau$, while $\Omega$ is invariant under $T^2$, where $T:\tau\to\tau+1$, since $q(\tau+2)=q(\tau).$ We thus understand the behavior of $\Omega$ under the group $\Gamma_\theta= \langle S, T^2 \rangle.$ Hence, it remains to understand the boundary behavior modulo $\Gamma_\theta,$ that is, the cusps $1$ and $\infty$ of $\Gamma_\theta$, which correspond to the growth condition and the analyticity at $q=0$ assumption made in Lemma \ref{lem: growth cond}, respectively (see the proof below).
\end{remark}

\begin{proof} 
Lemma \ref{lem: growth cond of omega} establishes the existence of the solution, so it suffices to prove uniqueness.

Suppose $\omega$ and $\omega_1$ both satisfy the conditions of Lemma \ref{lem: growth cond}. Let $\Delta \omega:= \omega-\omega_1,$ and assume for contradiction that $\Delta \omega \not\equiv 0$.  Then subtracting \eqref{eq:omega tlaw} for $\omega$ and $\omega_1$ gives the homogeneous equation 
\begin{equation}\label{eq:omega homog}
q^{2/3}\,\Delta \omega(-q) = -\sqrt{\frac{\pi}{\alpha}}\,q_1^{2/3}\,\Delta \omega(-q_1), \qquad \alpha>0.
\end{equation}
\begin{remark}
$\Delta \omega$ is holomorphic on $\mathbb D$ and still satisfies the bound in \eqref{eq:growth-omega}. 
\end{remark}
\noindent Consider the holomorphic function on $\mathbb H$
\begin{equation} \label{eq:Udef}
    U(\tau):=q^{2/3}\Delta \omega(-q), \qquad \tau\in\mathbb H
\end{equation}
and \eqref{eq:omega homog} gives\footnote{Both sides of \eqref{eq:omega homog} are holomorphic on $\{\alpha:\Re(\alpha)>0\},$ and they coincide for $\alpha\in \mathbb{R}^+$. Hence,  by the identity theorem,  \eqref{eq:omega homog} holds for all $\Re (\alpha)>0$.}
\begin{equation}\label{eq:U S}
    U(-1/\tau)=-\sqrt{-i\tau}\,U(\tau),\qquad \tau\in\mathbb H.
\end{equation}
\begin{remark} The principal branch $\sqrt{-i\tau}$ is  well-defined on $\mathbb H$ since $\Re(-i\tau)=\Im\tau>0$.
\end{remark}
\noindent Now define 
\begin{equation}
    A(\tau):=\frac{U(\tau)}{\vartheta_3(\tau)},
\end{equation}
where $\vartheta_3$ is the Jacobi theta function \cite{WW}. It is well-known that $ \vartheta_3$ has no zeros on $\mathbb H$ (see e.g. \cite[\href{https://dlmf.nist.gov/20.5.E9}{(20.5.9)}]{DLMF}) and satisfies the transformation laws \cite{WW}
\begin{equation}\label{eq:theta3 S}
    \vartheta_3(\tau+2)=\vartheta_3(\tau), \qquad \vartheta_3(-1/\tau)=\sqrt{-i\tau}\,\vartheta_3(\tau), \qquad \tau\in\mathbb H.
\end{equation}
Therefore, $A$ is holomorphic on $\mathbb H$, and using \eqref{eq:U S},\eqref{eq:theta3 S}, we observe that
\begin{equation} \label{eq: A under S}
A(-1/\tau)=\frac{U(-1/\tau)}{\vartheta_3(-1/\tau)} =\frac{-\sqrt{-i\tau}\,U(\tau)}{\sqrt{-i\tau}\,\vartheta_3(\tau)}=-A(\tau).
\end{equation}
On the other hand, by  Definition \ref{def:cusp} and \eqref{eq:Udef}, we have $U(\tau+2)=e^{\frac{4i\pi}{3}}U(\tau)$. Thus, by \eqref{eq:theta3 S}, we obtain 
\begin{equation}\label{eq: A under T2}
    A(\tau+2)=e^{\frac{4i\pi}{3}}A(\tau).
\end{equation}
Hence, by \eqref{eq: A under S},\eqref{eq: A under T2}, we see that the function
\begin{equation}\label{eq:h def}
h(\tau):=A(\tau)^6=\left(\frac{q^{2/3}\Delta \omega(-q)}{\vartheta_3(\tau)}\right)^6
\end{equation}
is holomorphic on $\mathbb H$ and invariant under the theta group $\Gamma_\theta$ (defined in Example \ref{ex:theta-group}), which has only two cusps: $1$ and $\infty$. Moreover, $h$ has the following behavior at the cusps.

\begin{lemma} If the function $h$, defined in \eqref{eq:h def}, is not identically zero, then it has the following behavior at the cusps of $\Gamma_\theta$.
    \begin{enumerate}
        \item[a.] $h$ is holomorphic at the cusp $\infty$ and $ \ord_\infty(h) \ge 4.$
        \item[b.] $h$ is meromorphic at the cusp $1$ and $\ord_1(h) \ge -1.$
\end{enumerate}
\label{lem: h beh}
\end{lemma}
\begin{proof}
The proof of Lemma \ref{lem: h beh} is a standard  calculation using Definition \ref{def:cusp}. The behavior at the cusp $\infty$ is controlled by the assumption that $\omega$ is analytic at $q=0$, while the behavior at the cusp $1$ is controlled by the growth condition in \eqref{eq:growth-omega}. We relegate the details to Appendix \ref{sec: append A}. 
\end{proof}
\begin{remark}
    Since $\Delta \omega \not\equiv 0$, we have $h \not\equiv 0,$ and it follows by Lemma \ref{lem: h beh} that $\ord_j(h)<\infty $ for $j \in \{1, \infty\}.$
\end{remark}
\begin{remark} Lemma \ref{lem: h beh}$(a)$ implies that $h(\infty)=0,$ so if $h$ is a constant function, we immediately get the contradiction $h \equiv 0$. We thus assume WLOG that $h$ is non-constant.
\end{remark}
\noindent Assume Lemma \ref{lem: h beh}. Since $h$ is $\Gamma_\theta$-invariant and is meromorphic at the cusps, it therefore extends to a non-constant meromorphic function on the compact Riemann surface  \cite[p.~13]{Duke} $ X_{\Gamma_\theta}:=\Gamma_\theta\backslash(\mathbb H\cup \Q \cup \{1, \infty\}),$ and we have \cite[cor. 10.22.]{Forster}
\begin{equation}\label{sum: ord}
    \sum_{P\in X_{\Gamma_\theta}}ord_P(h)=0.
\end{equation}
Since $h$ is holomorphic on $\mathbb H,$ we have $\ord_P(h) \ge 0$ for all $P\in X_{\Gamma_\theta} \setminus \{1, \infty\}$ \footnote{{Holomorphy of $h$ on $\mathbb H$ implies holomorphy on $X_{\Gamma_\theta} \setminus \{1, \infty\}$ (this follows from \cite[Lemma 4.11]{MilneMF} with $k=0,$ for example).}}. Therefore, if $h \not \equiv 0,$ \eqref{sum: ord} and Lemma \ref{lem: h beh} yield the contradiction 
\begin{equation}
0=\sum_{P\in X_{\Gamma_\theta}}\ord_P(h) \ge \ord_1(h)+\ord_\infty(h)\ge 3.
\end{equation}

This completes the proof of Lemma \ref{lem: growth cond}.
    \end{proof} Using Remark \ref{rem: omega f unique}, \Cref{T:omega uniqueness} follows.    \end{proof}
\begin{remark}
    { An alternative proof is given in \cite{CDGG}, \S 4.9.\footnote{We note two {small typos in \cite{CDGG},} which do not affect the flow of the arguments or the result: in (4.120) $q^2$ should be $q^4$, and in Theorem 4.116 the factor $\omega(q)\tilde{q}^{1/12}$ should read $\sqrt{\alpha}\,\omega(q)\tilde{q}^{1/12}$.} 
That approach, while effective in the cases of order 
3 (and 5) mock theta functions, relies on the use of Hauptmoduln, and it is not evident how it would extend to higher-order mfs. By contrast, the present strategy is both simpler and amenable to extension \cite{inprep}.}
\end{remark}

\section{Generalization to higher orders}

The methods presented in this work are part of a general uniqueness proof approach applicable to higher order mock theta functions. However, the higher order setting requires adaptations and the corresponding analysis becomes substantially more involved in the higher order cases. The higher order theory will appear in a separate paper \cite{inprep}.

\appendix
\label{sec:appendixa}

\section{Proof of Lemma \ref{lem: h beh}}
\label{sec: append A}
Recall the function $h$ defined in \eqref{eq:h def},
\begin{equation}
h(\tau):=A(\tau)^6=\left(\frac{q^{2/3}\Delta \omega(-q)}{\vartheta_3(\tau)}\right)^6,
\end{equation}
and the statement of Lemma \ref{lem: h beh}:

\begin{lemma} If the function $h$, defined in \eqref{eq:h def}, is not identically zero, then it has the following behavior at the cusps of $\Gamma_\theta$.
    \begin{enumerate}
        \item[a.] $h$ is holomorphic at the cusp $\infty$ and $ \ord_\infty(h) \ge 4.$
        \item[b.] $h$ is meromorphic at the cusp $1$ and $\ord_1(h) \ge -1.$
\end{enumerate}
\end{lemma}

\begin{proof}
For \textbf{(a)}, we study $h$ as $\tau \to i \infty$, i.e. as $q\to 0$ (recall the width of the cusp $\infty$ is $2$). Indeed, writing
\begin{equation}
    \vartheta_3(\tau)=1+O(q), \qquad \Delta \omega(q)=:a_m \, q^m +O(q^{m+1}),\qquad a_m \neq 0, \qquad q\to0,
\end{equation}
for some $m\ge0,$ we see that 
\begin{equation}
    h(\tau)=\left(\frac{q^{2/3}((-1)^m a_m \, q^m+O(q^{m+1}))}{1+O(q)}\right)^6=q^{6m+4}\bigl(a_m^6+O(q)\bigr), \qquad q\to0,
\end{equation}
which implies that $\ord_\infty(h)=6m+4\ge 4.$ This proves \textbf{(a)}. For \textbf{(b)}, choose
\begin{equation}
    \sigma_1:=\begin{pmatrix}1&-1\\ 1&0\end{pmatrix}\in SL_2(\Z), \qquad \tau:=\sigma_1(z)=\frac{z-1}{z}=1-\frac1z, 
\end{equation}
so that $\sigma_1(\infty)=1.$  By Definition~\ref{def:cusp}, it is enough to study the Laurent expansion of \(H(z):=h(\sigma_1 z)\) as \(\Im z\to\infty\). We begin by noting that the function $H(z)$ is holomorphic and $1$-periodic on $\mathbb H$. Indeed, letting
\begin{equation}
    \gamma_1:=\sigma_1 T\sigma_1^{-1}
    =
    \begin{pmatrix}0&1\\ -1&2\end{pmatrix}
    =
    -S(T^2)^{-1}\in \Gamma_\theta
\end{equation}
and using the fact that \(h\) is \(\Gamma_\theta\)-invariant, we observe that
\begin{equation}
    H(z+1)=h(\sigma_1(z+1))=h( \sigma_1 T z)= h(\gamma_1 \sigma_1 z) = h(\sigma_1 z)= H(z).
\end{equation}
   
Periodicity of $H$ then implies that there is a unique holomorphic function $F$ on the punctured unit disk
\begin{equation}
     \{u: 0<|u|<1\},
\end{equation}
such that $H(z)= F(e^{2 \pi iz})$ for all $z \in \mathbb H$ (see \cite[Prop.~12.1.3]{ENTStein}). Therefore, to prove \textbf{(b)}, it suffices to show that 
\begin{equation}
    |u\, F(u)|\lesssim 1 \qquad \text{as } u:=e^{2\pi i z}\to 0,
\end{equation}
since the latter implies that $u\, F(u)$ extends holomorphically to $u=0$ by Riemann's removable singularity theorem. Consequently, $h(\tau)= H(z)= F(u)$ is meromorphic at the cusp $1$ and $\ord_1(h)\ge -1$.

To prove this estimate, choose \(Y_0=Y_0(\theta_0)>0\) large enough so that for every
\begin{equation}
    z\in \mathcal S_{Y_0}:=\{x+iy: x\in[0,1],\ y\ge Y_0\},
\end{equation}
the quantity $\alpha_z:=\frac{i\pi}{z}$ satisfies the assumption $|\arg\alpha_z|\le \theta_0$ (possible since $0 \le |\arg \alpha_z| \le \arctan(1/\Im z)$). Now since $\tau=1-1/z,$ we have  
\begin{equation}
    -q(\tau)=-e^{i\pi(1-1/z)}=e^{-i\pi/z}=e^{-\alpha_z}.
\end{equation}
Accordingly, we write
\begin{equation}
       q_{1,z}:=\exp \Bigl(-\frac{\pi^2}{\alpha_z}\Bigr)=e^{i\pi z}, \qquad u:= e^{2\pi i z}= q_{1,z}^2.
\end{equation}
Applying the growth condition to $\Delta \omega,$ we get that for $z \in \mathcal S_{Y_0}$ and \(Y_0\) sufficiently large,
\begin{equation}\label{eq:Ubound 1}
|U(\tau)|
=|q(\tau)^{2/3}\Delta\omega(-q(\tau))| \le |\Delta\omega(e^{- \alpha_z})|
\lesssim |\alpha_z|^{-1/2}\,|q_{1,z}|^{-1/12}.
\end{equation}
On the other hand, using standard Jacobi theta transformation laws \cite[\href{https://dlmf.nist.gov/20.7.E31}{(20.7.31)}, \href{https://dlmf.nist.gov/20.7.E33}{(20.7.33)}]{DLMF} and the Jacobi triple product identity for $\vartheta_2$ \cite[\href{https://dlmf.nist.gov/20.4.E3}{(20.4.3)}]{DLMF}, we have 
\begin{equation}
    \vartheta_3(\tau)=\vartheta_3 \left(1-\frac1z\right)
=\vartheta_4 \left(-\frac1z\right)
=\sqrt{-iz}\,\vartheta_2(z)= 2 \sqrt{-iz}\,  q_{1,z}^{\,1/4}\prod_{n\ge 1}(1- q_{1,z}^{\,2n})(1+q_{1,z}^{\,2n})^2.
\end{equation}
Thus, since the last product is bounded away from $0$ uniformly on $\mathcal S_{Y_0}$ (with large enough $Y_0$), it follows that there is a constant $c>0$ such that 
\begin{equation}\label{eq:theta3 lower}
|\vartheta_3(\tau)|\ge c\,|z|^{1/2}\,|q_{1,z}|^{1/4}
= c\,|z|^{1/2}\,|u|^{1/8}
\end{equation}
holds uniformly on $\mathcal S_{Y_0}$. Hence using \eqref{eq:Ubound 1}, \eqref{eq:theta3 lower}, and $\alpha_z=\frac{i\pi}{z}$, we obtain  
\begin{equation}
    |A(\tau)|=\frac{|U(\tau)|}{|\vartheta_3(\tau)|}
\lesssim \frac{|\alpha_z|^{-1/2}|q_{1,z}|^{-1/12}}{|z|^{1/2}|q_{1,z}|^{1/4}}
\asymp |q_{1,z}|^{-1/3}=|u|^{-1/6}, \qquad u \to 0,
\end{equation}
so that 
\begin{equation}
    |F(u)|= |H(z)|=|h(\sigma_1 z)|=|h(\tau)|=|A(\tau)|^6\lesssim |u|^{-1}, \qquad u \to 0.
\end{equation}
This completes the proof of the lemma.
\end{proof}

\section{Action of $SL_2(\mathbb{Z})$ on $v$ in \eqref{eq:v23 def}}\label{S:SL2Z}
 This is a straightforward calculation. Since $S$ and $T$ generate SL$_2(\mathbb{Z})$,  Equations \eqref{eq:S v} and \eqref{eq:T v} reduce the action of SL$_2(\mathbb{Z})$ on $v$ to one in the group $G$ generated by $-M$ and $D$, defined in \eqref{eq:m} and \eqref{eq:T v}, respectively. A calculation shows that $G$ is the group presented by these generators, $D$ and $-M$ and the relations $D^{20}=M^{2}=(-MD)^{3}=\mathbb{1}$.

  A similar calculation applies to $\chi_0$ and $\chi_1$, after adding the nonholomorphic part, and using their transformations (see \cite[pp.110]{GM12}).

\section{Transformation identities not satisfying the conditions}
\label{sec:appendix-C}

For the order 5 mock theta functions in section \ref{sec:mf$_5$}, it is possible to combine known identities listed in \cite{GM12} to satisfy the mf$_5$ transformation identities \eqref{eq:mf5-identity} for $\alpha>0$:
    \begin{eqnarray}
       \mathcal K_0(Q)&=&-3\,\psi_0(-\sqrt{Q}) \qquad; \qquad \mathcal K_1(Q)=3\,\sqrt{Q}\, \psi_1(-\sqrt{Q})
       \label{eq:mf$_5$-other-solution1}\\
   \mathcal K_0(Q)&=&\frac{3}{2}\,f_0(-\sqrt{Q}) \qquad; \qquad \mathcal K_1(Q)=-\frac{3}{2}\,\sqrt{Q}\, f_1(-\sqrt{Q})
   \label{eq:mf$_5$-other-solution2}\\
   \mathcal K_0(Q)&=&\frac{3}{2}\left(\phi_0(-Q)-\psi_0(\sqrt{Q})\right) \quad; \quad \mathcal K_1(Q)=-\frac{3}{2}\left(\phi_1(-Q)+\sqrt{Q}\,\psi_1(\sqrt{Q})\right)
    \label{eq:mf$_5$-other-solution4}
    \end{eqnarray}
Here, $\psi_0$, $\psi_1$, $f_0$, $f_1$, 
$\phi_0$, $\phi_1$ are standard order 5 mock theta functions (\cite{GM12}, p 100).
However, none of the solutions in \eqref{eq:mf$_5$-other-solution1},\eqref{eq:mf$_5$-other-solution4} matches the unary $Q$-series structure on the Stokes line $\alpha <0$, and furthermore the solutions in \eqref{eq:mf$_5$-other-solution1},\eqref{eq:mf$_5$-other-solution4} are not holomorphic in $Q$.

Similarly, for the order 3 mock theta functions in section \ref{sec:mf$_3$},  a different solution to the transformation law \eqref{eq:omega tlaw} can be deduced by combining identities in \cite{GM12}, pages 114-115:
    \begin{eqnarray}
        \sqrt{\frac{12\alpha}{\pi}}\, W_3(\alpha) &=& q^{2/3}\, \left[ -\rho(-q) +\frac{1}{2}q^{-2/3}\, \xi(-q^{1/3})\right] \nonumber\\
        &&
        +\sqrt{\frac{\pi}{\alpha}}\, 
        q_1^{2/3}\, \left[ -\rho(-q_1) +\frac{1}{2}q_1^{-2/3}\, \xi(-q_1^{1/3})\right]
        \label{eq:mf$_3$-non}
    \end{eqnarray}
   Here the mf$_3$ $\rho$ and $\xi$ are defined as (\cite{GM12}, pp 100-101):\footnote{There is a small typo in \cite{GM12} for the definition of $\xi(q)$.}
    \begin{eqnarray}
       \rho(q)&:=& \sum_{n=0}^\infty \frac{q^{2n(n+1)}\, (q; q^2)_{n+1}}{(q^3; q^6)_{n+1}}
       \\
       \xi(q)&:=& 1+2\sum_{n=1}^\infty \frac{q^{6n(n-1)+1}}{(q; q^6)_{n}\,(q^5; q^6)_{n}}
       \label{eq:rho}     
    \end{eqnarray}
    However, the combination $(-\rho(-q) +\frac{1}{2}q^{-2/3}\, \xi(-q^{1/3}))$ in \eqref{eq:mf$_3$-non} does not match the structure on the Stokes line, $\alpha<0$, and furthermore it is not holomorphic in $q$. Indeed, from the identities listed in \cite[p. 108]{GM12}, we see that $(-\rho(-q) +\frac{1}{2}q^{-2/3}\, \xi(-q^{1/3}))$ differs from $\omega(-q)$ by the addition of a purely modular combination of $q$-Pochhammer symbols, which break the unary expansions at large $q$, and which are not all holomorphic in $q$.

\end{document}